\documentclass{amsart}

\usepackage[shortlabels]{enumitem}
\usepackage[colorlinks=true,linkcolor=blue,citecolor=blue,urlcolor=blue]{hyperref}

\newtheorem{lem}[equation]{Lemma}
\newtheorem{prop}[equation]{Proposition}

\newtheorem{longthm}[equation]{Longest Root Theorem}
\newtheorem{shortthm}[equation]{Shortest Root Theorem}

\theoremstyle{definition}
\newtheorem{eg}[equation]{Example}

\theoremstyle{remark}

\newcommand{\R}{\mathbb{R}}
\newcommand{\C}{\mathbb{C}}
\newcommand{\vol}{\operatorname{vol}}

\newcommand{\abs}[1]{|{#1}|}

\begin{document}

\title{The longest and shortest roots of a real cubic}
\author{Jason Bland}
\author{Skip Garibaldi}
\author{Joel Rosenberg}



\begin{abstract}
There are many formulas in the literature providing roots of a real cubic that avoid some of the well-known pathologies of Cardano's formulas.  Among these, we identify two that consistently provide the unique roots of a depressed cubic that have the greatest and smallest absolute value, whenever those exist.  We call these the longest and shortest roots.
The existence conditions are elementary  and are in terms of the signs of the coefficients and the discriminant.  Our proofs use two algebraic
identities satisfied by hypergeometric functions; once the standard
real branches are fixed, the root comparisons are entirely real.  As an
application, the longest-root formula gives an explicit factorization of all
but a vanishing proportion of depressed real quartics.
\end{abstract}

\maketitle

\section{Introduction}
Over the last 500 years, humanity has come up with various formulas for the roots of a cubic polynomial with real coefficients, see \cite{Fettis} or \cite{Tignol} for surveys.  The main point of this article is that two of those formulas, which appeared in an article by Zucker \cite{Zucker}, consistently pick out the roots of the cubic that have the greatest or least absolute value, whenever that makes sense.

Zucker expressed the three roots of 
\[
  f(t)=t^3+pt+q \quad \in\R[t],
\]
using formulas in which every displayed quantity is real---no
excursion into the complex plane, and at worst a single square root.  Write
\begin{equation}\label{disc}
  \Delta:=-4p^3-27q^2
\end{equation}
for the discriminant of $f$. For $p \ne 0$, put
\begin{equation}\label{A.def}
  A:=\frac{-\Delta}{4p^3}=1+\frac{27q^2}{4p^3}.
\end{equation}
His formulas are in terms of Gauss's
hypergeometric function
\begin{equation}\label{hyper.def}
 F(a,b;c;z)=1+\frac{ab}{c}\frac{z}{1!}
 +\frac{a(a+1)b(b+1)}{c(c+1)}\frac{z^2}{2!}+\cdots.
\end{equation}
The two roots we are most interested in are
\begin{align}
 \alpha(p,q)&:=\frac{3q}{p}
 F\left(\frac13,\frac23;\frac12;A\right) \quad \text{and}\label{alpha.def}\\
 \beta_+(p,q)&:=-\frac{q}{p}
 F\left(\frac23,\frac43;\frac32;
       \frac12(1-\sqrt A)\right).\label{beta.def}
\end{align}
The companion expression to \eqref{beta.def} with $1+\sqrt A$ instead of $1-  \sqrt{A}$ supplies the third root where that
branch is real.  

We will call one root \emph{longer} than another when it has larger absolute value,
and \emph{shorter} when it has smaller absolute value.  Our main results
identify the roots in \eqref{alpha.def} and \eqref{beta.def}.

\begin{longthm}\label{long.thm}
Let $f(t)=t^3+pt+q\in\R[t]$.  If $p \ge 0$ or $q = 0$, then no  root of $f$ is longer than the others.  Otherwise:
\begin{enumerate}[(a)]
\item \label{long.def} $\alpha(p,q)$ is defined and real;
\item \label{long.rt} $f(\alpha(p,q))=0$;
\item \label{long.long} $\alpha(p,q)$ is longer than the other two roots of $f$; and
\item \label{long.sign} $\alpha(p,q)$ has the same sign as $-q$.
\end{enumerate}
\end{longthm}

\begin{shortthm}\label{short.thm}
Let $f(t)=t^3+pt+q\in\R[t]$.  If $p \le 0$ and $\Delta\le0$, then no  root of $f$ is shorter than the
others.  Otherwise:
\begin{enumerate}[(a)]
\item \label{short.def} $\beta_+(p,q)$ is defined and real;
\item \label{short.rt} $f(\beta_+(p,q))=0$; and
\item \label{short.short} $\beta_+(p,q)$ is shorter than the other two roots of $f$.
\end{enumerate}
\end{shortthm}

Thus the two theorems, which appear to be new, give a simple classification.  If $p>0$, only a shortest
root exists.  If $p<0$ and $\Delta<0$, only a longest root exists.  If $p<0$,
$\Delta>0$, and $q\ne0$, both exist.  The loci $p=0$, $q=0$, and $\Delta=0$
account for the relevant ties and multiple roots.  For $t^3-15t-4$, both
theorems apply: $\alpha=4$ is the longest root and
$\beta_+=-2+\sqrt3\approx-0.27$ is the shortest.  (It is amusing to note that in this example, Cardano's method expresses the root $4$ as
$\sqrt[3]{2+11\sqrt{-1}}+\sqrt[3]{2-11\sqrt{-1}}$; see
\cite[\S2.3(c)]{Tignol}.)

Although \eqref{alpha.def} and \eqref{beta.def} are written using the
transcendental function $F$, the two functions involved satisfy polynomial identities.  Those identities make the root verifications short and allow the ordering of
the roots to be read off by elementary real arguments.  On their disks of
convergence, the formulas reduce respectively to a power series in the
discriminant studied in \cite{GenericRoot} and to the classical trinomial series
of Lambert and Lagrange from the 1700s (\autoref{series.sec}).  Finally, \autoref{quartic.sec} applies the
longest-root formula to factor real quartics.

\section{Two hypergeometric functions and their algebraic identities}
\label{hyper.sec}

Write
\[
 G(w):=F\left(\frac13,\frac23;\frac12;w\right),\qquad
 H(w):=F\left(\frac23,\frac43;\frac32;w\right).
\]
Throughout, $F$ denotes the standard branch obtained by analytic continuation
of the Maclaurin series \eqref{hyper.def} to the cut plane
$\C\setminus[1,\infty)$; see \cite[\S15.2(ii)]{DLMF}.  Its restriction to
$(-\infty,1)$ is real analytic and agrees there with the continuation from
$w=0$.  All later evaluations of $G$ and $H$ use this real branch.

Both identities below, and the two power series of Section~\ref{series.sec},
are governed by a single algebraic function.  Let $C(z)$ be the \emph{ternary
generating function}, the unique formal power series with
\begin{equation}\label{Catalan.eq}
  C(z)=1+zC(z)^3,\qquad C(0)=1.
\end{equation}
The defining equation has a combinatorial meaning: a ternary tree is either a
single leaf or a root joined to three ordered ternary subtrees, so if $C$
counts such trees by their number of internal nodes, it must satisfy
\eqref{Catalan.eq}.  Analytically, $C$ is the case $t=3$ of the generalized
binomial series $\mathcal B_t$ of \cite[\S5.4]{Concrete}, so $C=\mathcal B_3$;
by Lagrange inversion its coefficients are the order-$3$ Fuss--Catalan numbers
\begin{equation}\label{Catalan.coeffs}
  [z^n]C(z)=\frac{1}{3n+1}\binom{3n+1}{n}=1,1,3,12,55,\dots,
\end{equation}
OEIS sequence A001764 \cite{OEIS}.  Solving \eqref{Catalan.eq} for $z$ gives
\begin{equation}\label{z.of.C}
  z=\frac{C-1}{C^3},
\end{equation}
which we use repeatedly to pass from a rational expression in $C$ to an
algebraic identity in $z$.  The two series attached to $G$ and $H$ are both
elementary rational functions of $C$.

\begin{lem}\label{sT.lem}
Set $s(z):=\sum_{n\ge0}\binom{3n}{n}z^n$ and
$T(z):=\sum_{n\ge0}\binom{3n+1}{n}z^n$.  Then, as formal power series,
\begin{equation}\label{sT.rational}
  s=\frac{C}{3-2C}\qquad\text{and}\qquad T=\frac{C^2}{3-2C}.
\end{equation}
Consequently
\begin{align}
  (4-27z)\,s(z)^3&=1+3s(z),\label{s.identity}\\
  z(4-27z)\,T(z)^3&=T(z)-1.\label{T.identity}
\end{align}
\end{lem}

\begin{proof}
Write $c_n=[z^n]C=\frac{1}{3n+1}\binom{3n+1}{n}$, so that
$\binom{3n+1}{n}=(3n+1)c_n$ and $\binom{3n}{n}=(2n+1)c_n$, the latter from
$\binom{3n}{n}=\frac{2n+1}{3n+1}\binom{3n+1}{n}$.  Since
$[z^n](zC')=nc_n$, these read
\[
  s=C+2zC'\qquad\text{and}\qquad T=C+3zC'.
\]
Differentiating \eqref{Catalan.eq} gives $C'=C^3+3zC^2C'$, and eliminating
$z$ by \eqref{z.of.C} yields $C'=C^4/(3-2C)$.  Substituting this into the two
displays proves \eqref{sT.rational}.  Finally, putting \eqref{z.of.C} and
\eqref{sT.rational} into \eqref{s.identity} and \eqref{T.identity} makes each
side a single rational function of $C$; clearing denominators, the two
differences are the zero polynomial in $C$, so both identities hold as formal
power series in $z$.
\end{proof}

The two hypergeometric functions are these series in a rescaled variable.
Comparing coefficients in \eqref{hyper.def} and \eqref{Catalan.coeffs} gives
\begin{equation}\label{GH.series}
  G(w)=s(4w/27)\qquad\text{and}\qquad H(w)=T(4w/27)\qquad(|w|<1).
\end{equation}

\begin{lem}\label{G.identity}
On $(-\infty,1)$, the function $G$ satisfies
\begin{equation}\label{G.id}
  4(1-w)G(w)^3=1+3G(w).
\end{equation}
\end{lem}

\begin{proof}
Substituting $z=4w/27$ in \eqref{s.identity} and using $G(w)=s(4w/27)$ gives
\eqref{G.id} for $|w|<1$.  Both sides are analytic on $(-\infty,1)$ and agree
on a neighborhood of $0$, so they agree throughout.
\end{proof}

\begin{lem}\label{H.identity}
On $(-\infty,1)$, the function $H$ satisfies
\begin{equation}\label{H.id}
  16w(w-1)H(w)^3=27\bigl(1-H(w)\bigr).
\end{equation}
\end{lem}

\begin{proof}
Substituting $z=4w/27$ in \eqref{T.identity} and using $H(w)=T(4w/27)$ gives
\eqref{H.id} for $|w|<1$; analytic continuation along $(-\infty,1)$ gives the
stated identity.
\end{proof}

\begin{lem}\label{positive.branches}
For every $w<1$, one has
\[
  G(w)>0\qquad\hbox{and}\qquad H(w)>0.
\]
\end{lem}

\begin{proof}
Both functions are continuous on the connected interval $(-\infty,1)$ and
satisfy $G(0)=H(0)=1$.  Equations \eqref{G.id} and \eqref{H.id} show that $G(w)$ and $H(w)$ never vanish.  Neither function
can therefore change sign on $(-\infty,1)$, and both are positive there.
\end{proof}

The Shortest Root Theorem \ref{short.thm} is written to be parallel to the Longest Root Theorem \ref{long.thm}, except that it is missing part \ref{long.sign}.  It is missing for a good reason: The shortest root 
has the sign of $-q$ on $\{p>0\}$ and the sign of $q$ on
$\{p<0,\Delta>0\}$.  This follows from Lemma~\ref{positive.branches} and the
prefactors in \eqref{alpha.def} and \eqref{beta.def}.

\section{The longest root}

We first record where the arguments of the hypergeometric functions lie.

\begin{lem}\label{arguments.lem}
Suppose $p \ne 0$ and let $A=-\Delta/(4p^3)=1+27q^2/(4p^3)$ as in \eqref{A.def}.
\begin{enumerate}[(i)]
\item \label{arguments.long} If $p<0$ and $q\ne0$, then $A<1$.
\item \label{arguments.short} If $p>0$ or $\Delta>0$, then $A\ge0$, and
$\frac12(1-\sqrt A)\le\frac12<1$.
\end{enumerate}
\end{lem}

\begin{proof}
If $p<0$, then $27q^2/(4p^3)\le0$, with equality only for $q=0$, proving
\ref{arguments.long}.

For \ref{arguments.short}, we show $A\ge0$ in each of the two cases; the bound
$\frac12(1-\sqrt A)\le\frac12<1$ is then immediate, since $A\ge0$ makes
$\sqrt A$ real and nonnegative.  If $p>0$, then $27q^2/(4p^3)\ge0$, so
$A=1+27q^2/(4p^3)\ge1\ge0$.  If instead $\Delta>0$, then \eqref{disc} gives
$4p^3=-\Delta-27q^2<0$, hence $p<0$; both the numerator $-\Delta$ and the
denominator $4p^3$ of $A=-\Delta/(4p^3)$ are then negative, so $A>0$.  In
either case $A\ge0$, proving \ref{arguments.short}.
\end{proof}

\begin{proof}[Proof of the Longest Root Theorem~\ref{long.thm}]
Suppose first that we are in the ``otherwise'' case, meaning that $p < 0$ and $q\ne0$.  By Lemma~\ref{arguments.lem}, $A<1$, so
\[
  \alpha:=\frac{3q}{p}G(A)
\]
is real, i.e., \ref{long.def}.

We verify that it is a root, \ref{long.rt}.  From \eqref{disc} and \eqref{A.def},
\begin{equation}\label{one.minus.A}
  1-A=1+\frac{\Delta}{4p^3}
      =\frac{4p^3+\Delta}{4p^3}
      =\frac{-27q^2}{4p^3}.
\end{equation}
Writing $G=G(A)$ and using \eqref{G.id},
\[
 \alpha^3
 =\frac{27q^3}{p^3}G^3
 =\frac{27q^3}{p^3}\frac{1+3G}{4(1-A)}
 =-q(1+3G).
\]
Consequently
\[
 f(\alpha)=-q(1+3G)+p\frac{3q}{p}G+q=0.
\]

By Lemma~\ref{positive.branches}, $G(A)>0$.  Since $p<0$, it follows at once
that $\alpha$ has the sign of $-q$, \ref{long.sign}.

For \ref{long.long}, suppose first that $\Delta>0$ and $p \ne 0$, so the three roots are distinct and real.  If
$q>0$, their product is $-q<0$, so there is one negative root and two positive
roots.  Since their sum is zero, the absolute value of the negative root is
the sum of the two positive roots and is therefore strictly larger than the
absolute value of either.  Thus the unique root with sign $-q$ is the longest.
The case $q<0$ is obtained by reversing all signs.  Hence $\alpha$ is the
longest root.

If $\Delta=0$, then
\begin{equation}\label{double.factor}
 f(t)=\left(t-\frac{3q}{p}\right)
      \left(t+\frac{3q}{2p}\right)^2.
\end{equation}
Here $A=0$ and $G(0)=1$, so
$\alpha=3q/p$, the simple root.  Its modulus is twice that of the double root
$-3q/(2p)$, and it has the sign of $-q$.

Finally, suppose $\Delta<0$.  Write the roots as a real number $z$ and a
complex-conjugate pair $u,\bar u$.  Since $u+\bar u=-z$,
\begin{equation}\label{modulus.identity}
 p=u\bar u+z(u+\bar u)=|u|^2-z^2.
\end{equation}
As $p<0$, this gives $|z|>|u|$.  The real root is therefore the longest, and
$\alpha=z$.  Also $q=-z|u|^2$, so $z$ has the sign of $-q$.

For the converse, if $p>0$ then $\Delta<0$, and
\eqref{modulus.identity} gives $|z|<|u|$; the two nonreal roots tie, so no
root is longest.  If $p=0$, the three cube roots of $-q$ have equal modulus.
If $q=0$, the roots are $0$ and $\pm\sqrt{-p}$ when $p<0$, or $0$ and a
conjugate pair of equal modulus when $p>0$; again no root is uniquely longest.
\end{proof}

\section{The shortest root}

\begin{proof}[Proof of the Shortest Root Theorem~\ref{short.thm}]
Suppose first that we are in the ``otherwise'' case, meaning that (i) $p > 0$ or (ii) $p \ne 0$ and $\Delta > 0$.
By Lemma~\ref{arguments.lem}, $A\ge0$ and
\[
 w:=\frac12(1-\sqrt A)<1.
\]
Thus
\[
 \beta_+:=-\frac{q}{p}H(w)
\]
is real, \ref{short.def}.  

Since $\sqrt A=1-2w$, one has
\[
 \frac{27q^2}{4p^3}=A-1=(1-2w)^2-1=4w(w-1),
\]
so
\begin{equation}\label{ww.relation}
 w(w-1)=\frac{27q^2}{16p^3}.
\end{equation}
Writing $H=H(w)$ and using \eqref{H.id} and \eqref{ww.relation},
\[
 27(1-H)=16w(w-1)H^3=\frac{27q^2}{p^3}H^3.
\]
Thus $q^2H^3=p^3(1-H)$, and
\[
 \beta_+^3=-\frac{q^3}{p^3}H^3=q(H-1).
\]
It follows that
\[
 f(\beta_+)=q(H-1)-qH+q=0,
\]
verifying \ref{short.rt}.

For \ref{short.short}, if $p>0$, then $\Delta<0$ and the roots are a real root $z$ and a conjugate
pair $u,\bar u$.  Equation \eqref{modulus.identity} gives $|z|<|u|$, so the
real root is shortest.  Since $\beta_+$ is real, it equals $z$.  

If $\Delta>0$ and $p \ne 0$, then $p<0$ and the three roots are real.  If $q=0$, then
$\beta_+=0$, which is plainly the shortest root among
$0,\pm\sqrt{-p}$.  Assume $q\ne0$.  On each of the two connected regions
\[
 \{(p,q):p<0,\ \Delta>0,\ q>0\}
 \quad\hbox{and}\quad
 \{(p,q):p<0,\ \Delta>0,\ q<0\},
\]
the value $\beta_+(p,q)$ is a continuous real root of $f$.  No two roots can
have the same modulus there: two real roots of equal modulus would be
opposites, forcing the third root, and hence $q$, to be zero.  Therefore the
rank of $\beta_+$ among the three roots ordered by modulus is constant on
each region.  Keeping $p<0$ fixed and letting $q\to0$ with fixed sign, one
has $A\to1$, $w\to0$, and $H(w)\to H(0)=1$, so
$\beta_+(p,q)\to0$, while the other two roots tend to
$\pm\sqrt{-p}$.  Hence $\beta_+$ is the shortest root throughout both
regions.  

In the excluded region $p<0$ and $\Delta<0$, the unique real root is longest
by the proof of Theorem~\ref{long.thm}, while the conjugate roots tie in
modulus; hence no root is uniquely shortest.  If $p<0$ and $\Delta=0$, the
factorization \eqref{double.factor} shows that the double root
$-3q/(2p)$ has smaller modulus than the simple root but occurs twice, so no
root is strictly shorter than both others.  If $p = 0$, then all roots have the same length.
\end{proof}

\section{Non-depressed cubics}
\label{nondepressed.sec}

The theorems above are stated for depressed cubics, but they apply to any
cubic after a translation.  Given
\[
  g(t)=t^3+c_1t^2+c_2t+c_3\in\R[t],
\]
the substitution $t=x-c_1/3$ removes the quadratic term and produces the
depressed cubic
\[
  g\!\left(x-\frac{c_1}{3}\right)=x^3+px+q,
  \qquad
  p=c_2-\frac{c_1^2}{3},
  \quad
  q=\frac{2c_1^3}{27}-\frac{c_1c_2}{3}+c_3 .
\]
This translation shifts every root by the same amount $c_1/3$ and leaves the
discriminant unchanged.  The longest and shortest roots of $f$ correspond to the roots of $g$ that are farthest from or closest to the centroid $-c_1/3$ of the roots of $g$.

\begin{eg}[A cubic with a double root]
Take $g(t)=t(t-r)^2$ with $r\ne0$, whose roots are $0$ and the double root
$r$.  Expanding, $c_1=-2r$, $c_2=r^2$, $c_3=0$, so the centroid is
$-c_1/3=2r/3$, and the depression $t=x+2r/3$ yields
\[
  x^3+px+q,\qquad p=-\frac{r^2}{3},\quad q=\frac{2r^3}{27}.
\]
Here $p<0$ and $q\ne0$, so a longest root exists.  The depressed roots are
$-2r/3$ (simple) and $r/3$ (double).  The 
longest of these is $-2r/3$, which translates to the root $0$ of $g$, the root farthest from the centroid $2r/3$.
\end{eg}

\begin{eg}[Scaling]\label{scaling.eg}
For nonzero $m\in\R$ and
$g(t)=t^3+c_1t^2+c_2t+c_3$, put $h(t)=g(mt)/m^3$.  If $p_g,q_g$ and
$p_h,q_h$ are the respective depressed parameters, then
$p_g=m^2p_h$, $q_g=m^3q_h$, and $\Delta_g=m^6\Delta_h$.  Hence
$A=-\Delta/(4p^3)$ is unchanged, while the distinguished roots scale by
$m$.
\end{eg}

\section{The underlying series: discriminant and trinomial}
\label{series.sec}

In the region where the defining series \eqref{hyper.def} converges, Zucker's
functions reduce to two power-series solutions of the cubic.

\subsection*{The discriminant series.}
By \eqref{GH.series}, $G(A)=s(4A/27)$, where $s(z)=\sum_{n\ge0}\binom{3n}{n}z^n$
is the companion tree series of Lemma~\ref{sT.lem}.
Since $4A/27=-\Delta/(27p^3)$, formula \eqref{alpha.def} becomes
\begin{equation}\label{discriminant.series}
 \alpha(p,q)=\frac{3q}{p}\,
 s\!\left(\frac{-\Delta}{27p^3}\right)
 =\frac{3q}{p}
 \sum_{n\ge0}\binom{3n}{n}
 \left(\frac{-\Delta}{27p^3}\right)^n,
\end{equation}
a series that appeared naturally in \cite{GenericRoot}.  The hypergeometric series for $G(A)$ converges absolutely when
$|A|<1$,
which is equivalent to $|\Delta|<4|p|^3$.  Equivalently, Stirling's formula
shows that the series in \eqref{discriminant.series} has radius
$4/27$ in the variable $-\Delta/(27p^3)$.  Its coefficients
$\binom{3n}{n}=1,3,15,84,\ldots$ form OEIS sequence A005809
\cite{OEIS}.

\subsection*{The trinomial series.}
For $\beta_+$, tracing through the proof of the Shortest Root Theorem, we have
\[
w(w-1) = -\frac{27}{16} z \quad \text{for} \quad z = \frac{q^2}{(-p)^3}.
\]
Plugging that into \eqref{H.id} gives $zH(w)^3 = H(w) - 1$.  Because $H(w) \to 1$ as $z \to 0$, $H(w)$ is $C(z)$ from \eqref{Catalan.eq}, which gives:
\begin{equation}\label{trinomial.series}
 \beta_+(p,q)=-\frac{q}{p}
 \,C\!\left(\frac{q^2}{(-p)^3}\right).
\end{equation}
This is the classical series solution of a trinomial going back to Lambert
\cite[\S\S38--40]{Lambert} and Lagrange \cite[No.~12]{Lagrange:res}; see also
$X_{1,-1}$ in \cite{Sturmfels} or Th.~10 in \cite{WildRubine}.

\medskip
The point of Zucker's hypergeometric forms is that the chosen branches
continue these series past their disks of convergence.
\section{Factoring real quartic polynomials}\label{quartic.sec}

Every real quartic is a product of two real quadratics, but exhibiting the
factorization for a given quartic requires a positive root of an auxiliary
cubic resolvent.  We show that applying the longest-root formula to the
\emph{depressed} resolvent and then translating back supplies such a root for all but a negligible fraction of depressed quartics, thereby providing an explicit factorization of those quartics.

Let
\begin{equation}\label{quartic}
 g(t)=t^4+ct^2+dt+e
\end{equation}
be a depressed quartic.  Its \emph{Descartes cubic resolvent} is
\begin{equation}\label{resolvent}
 g_3(t)=t^3+2ct^2+(c^2-4e)t-d^2.
\end{equation}
The classical link between the roots of $g_3$ and factorizations of $g$ is the following; see
\cite[Appendix, \S15]{Chrystal} or \cite{Brookfield}.

\begin{lem}\label{quartic.factor.lem}
If $g_3(\gamma^2)=0$ for some real $\gamma\ne0$, then
\[
 g(t)=\left(t^2+\gamma t+
 \frac{\gamma^3+c\gamma-d}{2\gamma}\right)
 \left(t^2-\gamma t+
 \frac{\gamma^3+c\gamma+d}{2\gamma}\right).
\]
\end{lem}

To apply the Longest Root Theorem~\ref{long.thm} we depress \eqref{resolvent}
as in \autoref{nondepressed.sec}.  The resolvent has quadratic coefficient
$2c$, so its centroid is $-2c/3$, and the substitution
\begin{equation}\label{resolvent.shift}
  t=x-\frac{2c}{3}
\end{equation}
yields
\begin{equation}\label{depressed.resolvent}
 g_3\left(x-\frac{2c}{3}\right)=x^3+p_3x+q_3,
\end{equation}
with
\begin{equation}\label{p3q3}
 p_3=-\frac{c^2}{3}-4e,
 \qquad
 q_3=-\frac{2c^3}{27}+\frac{8ce}{3}-d^2.
\end{equation}
Whenever $p_3<0$ and $q_3\ne0$, let $\alpha_3(c,d,e)$ be the longest root of
the depressed resolvent \eqref{depressed.resolvent}---the root of $g_3$
farthest from its centroid---and translate it back by setting
\begin{equation}\label{Phi.def}
  \Phi(c,d,e):=\alpha_3(c,d,e)-\frac{2c}{3}.
\end{equation}
By \eqref{resolvent.shift}, $\Phi$ is a root of the original resolvent
$g_3$.

For a depressed quartic \eqref{quartic}, define its \emph{naive height} to be
$\max\{\abs{c},\abs{d},\abs{e}\}$, following \cite[p.~2]{Borwein}.  The idea behind the next result is that, if we look at a large enough cube of depressed quartics (meaning all quartics with naive height at most some large $h$), then almost all of them have a resolvent cubic with a positive root and therefore an explicit factorization.

\begin{prop}\label{quartic.density}
For $h>0$, let $S_h\subseteq[-h,h]^3$ be the set of $(c,d,e)$ with $d\ne0$,
$p_3<0$, and $q_3\ne0$ for which the translated distinguished root
$\Phi(c,d,e)$ in \eqref{Phi.def} is positive.  Then
\[
  \frac{\vol(S_h)}{(2h)^3}\longrightarrow1
  \qquad(h\to\infty).
\]
\end{prop}

\begin{proof}
At the points $P_\pm=(\pm1,0,0)$, the resolvent is
\[
 g_3(t)=t(t\pm1)^2.
\]
Equation \eqref{p3q3} gives $p_3=-1/3$ and
$q_3=\mp2/27$ at these points, so the longest-root formula is defined in
neighborhoods of both.
After the shift \eqref{resolvent.shift}, the roots $0,\pm1,\pm1$ of $g_3$ correspond to $\pm 2/3,\mp 1/3,\mp 1/3$ for $P_\pm$.  Since $\pm 2/3$ is the longest of these, $\Phi(P_\pm)=0$.

We next observe that $\Phi > 0$ on a fixed neighborhood of $P_\pm$.
The distinguished root $\Phi$ varies continuously with $(c,d,e)$ on the
open set $p_3<0$, $q_3\ne0$.  By continuity of the roots, there is
$\varepsilon>0$ such that on
\[
 U_\pm:=\{|c\mp1|<\varepsilon,
           |d|<\varepsilon,
           |e|<\varepsilon\}
\]
the two roots of $g_3$ other than $\Phi$ remain within $1/2$ of $\mp1$.
Their product is therefore positive.  Since the product of all three roots
of $g_3$ is $d^2$, it follows that
$\Phi>0$ on $U_\pm\cap\{d\ne0\}$.

Next we observe that scaling can move $g$ into $U_+ \cup U_-$.
For $m>0$, put $f(t)=g(mt)/m^4$.  Then
\[
 f_3(t)=\frac{g_3(m^2t)}{m^6}.
\]
Depressing the two resolvents is compatible with this scaling, and
Example~\ref{scaling.eg} shows that their distinguished roots differ by the
positive factor $m^2$.  Thus the sign of $\Phi$ is scale invariant.  Taking
$m=\abs{c}^{1/2}$, a point with $c\ne0\ne d$ lies in $S_h$ as soon as
\[
 \left(\frac{c}{\abs{c}},
       \frac{d}{\abs{c}^{3/2}},
       \frac{e}{c^2}\right)
 \in U_+\cup U_-.
\]

Fix $\eta\in(0,1)$.  The slab $\abs{c}\ge\eta h$ fills a proportion $1-\eta$ of
$[-h,h]^3$, and on it
\[
 \frac{\abs{d}}{\abs{c}^{3/2}}\le\eta^{-3/2}h^{-1/2},
 \qquad
 \frac{\abs{e}}{c^2}\le\eta^{-2}h^{-1}.
\]
Both quantities are less than $\varepsilon$ for large $h$.  Apart from the
measure-zero sets $d=0$ and $q_3=0$, every point of this slab therefore
normalizes into $U_+\cup U_-$ and lies in $S_h$.  Hence
\[
 \liminf_{h\to\infty}\frac{\vol(S_h)}{(2h)^3}\ge1-\eta.
\]
As $\eta$ was arbitrary, the limit is $1$.
\end{proof}

\subsection*{Acknowledgements} We thank Sateesh Mane and Seiichi Manyama for helpful pointers and suggestions, and J-P.~Serre for posing the problem that eventually led to this paper.  

\bibliographystyle{amsalpha}
\bibliography{skip_master,cubic}

\end{document}